\documentclass[10pt,reqno]{amsart}

\usepackage{amssymb, amsthm, amsmath, latexsym}
\usepackage{amsfonts}
\usepackage{color,mathrsfs,epsfig}
\newcommand{\bp}{\begin{proof}}
\newcommand{\ep}{\end{proof}}
\newcommand{\be}{\begin{equation}}
\newcommand{\ee}{\end{equation}}

\newcommand{\e}{\varepsilon}

\newcommand{\Q}{Q}

\newcommand{\R}{\mathbb R}
\newcommand{\wcs}{W_{cs}}
\newcommand{\bwcs}{\bar W_{cs}}
\newcommand{\pcs}{\Phi_{cs}}
\newcommand{\zz}{\zeta}
\newcommand{\Ve}{V^{\e}}
\newcommand{\wcse}{\wcs^{\e}}
\newcommand{\wcsu}{W^{\e}_{cs 1 }}
\newcommand{\wcsd}{W^{\e}_{cs 2 }}
\newcommand{\wcsz}{W^0_{cs}}
\newtheorem{thm}{Theorem}[section]
\newtheorem{lemma}[thm]{Lemma}

\newtheorem{prop}[thm]{Proposition}

\numberwithin{equation}{section}

\theoremstyle{definition}
\newtheorem{remark}{Remark}[section]
\renewenvironment{proof}[1][Proof]{\textbf{#1.} }{\ \rule{0.5em}{0.5em}}

\numberwithin{equation}{section}

\begin{document}
\bibliographystyle{plain}
\title[Boundary layers for the self-similar viscosity ]{Boundary layers for self-similar viscous approximations of nonlinear hyperbolic systems}
\author{Cleopatra Christoforou \and Laura V.  Spinolo}

\address{C.C.: 
Department of Mathematics and Statistics\\
University of Cyprus\\
1678 Nicosia, Cyprus}
\email{Chistoforou.Cleopatra@ucy.ac.cy}
\urladdr{http://www2.ucy.ac.cy/$\sim$kleopatr}

\address{L.V.S.: Institut f\"ur Mathematik \\
Universit\"at Z\"urich,
Winterthurerstrasse 190,
CH-8057 Z\"urich, 
Switzerland }
\email{laura.spinolo@math.uzh.ch}
\urladdr{http://user.math.uzh.ch/spinolo/}

\keywords{hyperbolic systems; boundary Riemann problem; self-similar viscous approximation; vanishing viscosity; boundary layer.}

\subjclass[2000 ]{35L65, 35L50, 35K51.}
%
% Abstract

\dedicatory{ Dedicated to Professor Constantine M. Dafermos on the occasion of his 70th birthday}

\begin{abstract}
We provide a precise description of the set of residual boundary conditions generated by the self-similar viscous approximation introduced by Dafermos et al. We then apply our results, valid for both conservative and non conservative systems, to the analysis of the boundary Riemann problem and we show that, under appropriate assumptions, the limits of the self-similar and  the classical vanishing viscosity approximation coincide. We require neither genuinely nonlinearity nor linear degeneracy of the characteristic fields.   
\end{abstract}

\maketitle
%%%%%%%%%%%%%%%%%%%%%%%%%%%%%%%%%%%%%%%%%%%%%%%%%%
\section{Introduction and main results}\label{S: Introduction}
We are interested in systems of conservation laws in one space dimension
\be
\label{e:cl}
        \partial_t U + \partial_x \big[ F(U) \big] =0,
\ee
where  $U= U(t, x)$ attains values in $\R^n$, $t \ge 0$ and $x \in \R$ are the scalar independent variables and the flux function $F: \R^n \to \R^n$ is regular. System~\eqref{e:cl} is strictly hyperbolic if the Jacobian matrix $DF(U)$ has $n$ real, distinct eigenvalues 
\be
\label{sh}
     \lambda_1 (U) < \lambda_2 (U) < \dots < \lambda_n (U) 
\ee     
for every $U \in \R^n$. We refer to the books by Dafermos~\cite{D} and Serre~\cite{Serre:book} for an exposition of the current state of the theory of hyperbolic conservation laws. 

In the present paper, we characterize the set of residual boundary conditions generated by the self-similar viscous approximation introduced by Dafermos~\cite{D1} et al. and we then apply our results, valid for both conservative and non conservative systems, to the analysis of the so-called boundary Riemann problem. 

The theory of systems of conservation laws poses several challenges: in general classical solutions of Cauchy problems starting out from smooth initial data develop discontinuities and break down 
in finite time. Thus, one needs to look for distributional solutions. However, distributional solutions are not unique and therefore, various admissibility criteria, often motivated by physical considerations, have been introduced in the attempt at singling out a unique solution. We refer the reader to Dafermos~\cite{D}
for an extended discussion on admissibility criteria. The Cauchy problem obtained by coupling~\eqref{e:cl} with the initial data $U(0, x) = U_0(x)$ has been studied extensively. Existence and uniqueness results for global in time, admissible distributional solutions have been established via the random choice method of Glimm~\cite{Glimm} and  the front tracking algorithm, see Bressan et al~\cite{Bre:book} and Holden and Risebro~\cite{HoldenRisebro}. These results hold true under the assumption that the total variation $\mathrm{TotVar} \, U_0$ of the initial data is sufficiently small. If this condition is violated, then the oscillation or total variation of the solution  may  blow up in finite time. Existence and uniqueness results for data with large total variation have been achieved by imposing additional conditions on the structure of the flux function $F$. In the present paper, we only require that the flux $F$ satisfies the strict hyperbolicity condition~\eqref{sh} and hence, we focus on data of sufficiently small total variation.

In view of applications coming from physics, it would be natural to construct solutions of~\eqref{e:cl} by taking the limit $\varepsilon \to 0^+$ of the viscous approximation
\be
\label{e:va}
                   \partial_t Z^{\varepsilon} + 
            \partial_x \big[ F(Z^{\varepsilon}) \big] =
            \varepsilon \partial_x \big[ B(Z^{\varepsilon}) 
            \partial_x Z^{\varepsilon}  \big]. 
\ee 
Here, $B(Z)$ is an $n \times n$ matrix which depends on the physical model under consideration. Convergence results have been established in specific cases, see Dafermos~\cite{D} and the references therein. In particular, Bianchini and Bressan~\cite{BiaBrevv} proved convergence in the case when $B \equiv I$, while the proof of the convergence in the general physical case is still a challenging open problem. 

A key tool in the analysis of the system of conservation laws~\eqref{e:cl} is the study of self-similar solutions in the form $U(t, x)= \Q(x/t)$. Indeed, the aforementioned random choice and  front-tracking schemes are implemented by constructing approximate solutions that locally have the structure of a self-similar function. Also, the analysis of self-similar solutions provides information on both the local (in space-time) and the long-time behavior of the solution of a general Cauchy problem. Self-similar solutions and the so-called Riemann problem have been extensively studied, see in particular the works by Lax~\cite{lax}, Liu~\cite{Liu:rie, Liu:adm}, Tzavaras~\cite{Tz} and Bianchini~\cite{Bia:riemann} in the context of data of small total variation. 

To investigate the structure of self-similar solutions, Dafermos~\cite{D1}, Kalasnikov~\cite{kal} and Tupciev~\cite{Tu1966} independently introduced the self-similar viscous approximation 
\be
\label{e:ssva}
            \partial_t U^{\varepsilon} + 
            \partial_x \big[ F(U^{\varepsilon}) \big] =
            \varepsilon \,  t \, \partial_x \big[ B(U^{\varepsilon}) 
            \partial_x U^{\varepsilon}  \big]. 
\ee 
Indeed, because of the $``t"$ factor in front of the second-order term, the viscous approximation~\eqref{e:ssva} admits self-similar solutions $U^{\varepsilon} (t, x) = Q^{\varepsilon} (x/t)$ which satisfy the ordinary differential equation
\be
\label{e:visode}
              \e \frac{d }{ d \xi } \Big[ B (Q^\e) \frac{d Q^\e }{ d \xi } \Big] =   \frac{d }{d \xi} \big[ F(Q^\e) \big] - \xi \frac{d Q^\e}{ d \xi}\;.
\ee
Convergence results for~\eqref{e:visode} have been established, under suitable conditions imposed on the matrix $B$, by various authors, see in particular Tzavaras~\cite{Tzavaras:JDE,Tz}, Andreianov~\cite{Andreianov}, Joseph and LeFloch~\cite{JoLeF:nc} and the analysis and the references in Dafermos~\cite[Section 9.8]{D}.

In this paper, we consider the initial-boundary value problem obtained by focusing on the domain $(t, x) \in [0, + \infty [ \times [0, + \infty[$. The initial-boundary value problem poses all the challenges of the Cauchy problem that have been mentioned (breakdown of classical solutions, non uniqueness of distributional solutions), as well as additional difficulties due to the presence of the boundary. For instance, consider the problem obtained by coupling the viscous approximation~\eqref{e:va} with the Cauchy and Dirichlet data 
\be
\label{e:cdd}
          Z^{\varepsilon} (0, x) = U_0  \qquad Z^{\varepsilon} (t, 0) = U_b, 
\ee
where $U_0$ and $U_b$ are given constant states in $\R^n$. Since we are interested in small total variation solutions, we focus on the case when $|U_0 - U_b|$ is sufficiently small. If the matrix $B$ is singular (which is the case for most of the physically relevant systems), then problem~\eqref{e:va},~\eqref{e:cdd} may be ill-posed, namely admit no solutions. However, one can obtain a well-posed problem, at the price of higher technicalities, by relying on a more complicated formulation of the boundary condition (see Bianchini and Spinolo~\cite{BiaSpi}). To simplify the exposition, in the following we assume that problem~\eqref{e:va},~\eqref{e:cdd} is well-posed, but our considerations can be extended to the more general case of a suitable class of singular matrices considered in~\cite{BiaSpi}.

We consider the family of initial-boundary value problems obtained by coupling~\eqref{e:va} and~\eqref{e:cdd} and we assume that $Z^{\varepsilon}$ converge as $\varepsilon \to 0^+$ to a limit $U$ in  a suitable topology. Having data~\eqref{e:cdd}, one expects that the solution $U$ admits the self similar form $U(t, x) = \Q(x/t)$ for some function $\Q$ and, hence, one recovers in the limit a solution to the so-called \emph{boundary Riemann problem}. Moreover, because of the assumption that $|U_0 - U_b|$ is sufficiently small, one expects that the trace $\bar U\doteq\lim_{\xi \to 0^+} \Q(\xi)$ is well defined. It should be mentioned that such results have been established under additional assumptions, see for instance the analysis in Gisclon~\cite{Gis} and in Ancona and Bianchini~\cite{AnBia}. We emphasize that, in general, one has 
$$
    \bar U\doteq\lim_{x \to 0^+} U(t, x) = \lim_{\xi \to 0^+} \Q(\xi) \neq U_b 
$$
and, also, the limit $U$ of~\eqref{e:va},~\eqref{e:cdd} varies if the matrix $B$ varies, and we refer to the articles of Gisclon and Serre~\cite{Gis, GisSerre} for these observations. 

To investigate the relation between the boundary data $U_b$ and the so-called \emph{residual boundary condition} $\bar U\doteq\lim_{\xi \to 0^+} \Q(\xi)$, we focus on the \emph{non characteristic case}, namely we assume that, for every $U\in\R^n$, all the eigenvalues of the Jacobian matrix $DF(U)$ are bounded away from zero: 
\be 
\label{e:i:sesp}
         \lambda_1 (U) < \dots < \lambda_k (U) < -c < 0 < c < \lambda_{k+1 } (U) < \dots < \lambda_n (U)
\ee
for some integer $1\le k \leq n-1$ and for a suitable constant $c >0$.  Then, one expects (see again~\cite{AnBia, Gis} for rigorous results in specific cases) that there exists a function $ V$, the so-called boundary layer, which solves the following system: 
\be
\label{e:bl}
\left\{
\begin{array}{ll}
          B(V) V'= F(V)- F(\bar U) \\
            V(0) = U_b \qquad \lim_{y \to + \infty} V(y) = \bar U\;, 
\end{array}
\right.
\ee
where $ V'$ denotes the first derivative of the function $ V$. 

Regarding the self-similar viscous approximation of a boundary Riemann problem obtained by coupling equation~\eqref{e:ssva} with the data
\be
\label{e:cdd1}
       U^{\varepsilon} (0, x) = U_0  \qquad  U^{\varepsilon} (t, 0) = U_b, 
\ee
Joseph and LeFloch established compactness and convergence results in the case when $B$ is the identity matrix in~\cite{JoLeF:sv} or close to the identity in~\cite{JoLeF:nc}.  Moreover, they described the self-similar limit $U(t, x) = \Q(x/t)$ and showed that $\Q$ has bounded total variation. Hence, in particular, the trace $\lim_{\xi \to 0^+} \Q(\xi)$, which we denote by $\bar U$ again, is well-defined. Furthermore, they investigated the boundary layer and proved in Theorem $4.2$ of~\cite{JoLeF:sv} that there exists a boundary layer $V$ satisfying~\eqref{e:bl}. Their analysis of the boundary layer involves delicate manipulations of the equations that rely on the conservative form of~\eqref{e:ssva}.

In this paper, we provide a different approach to the analysis of the boundary layers of the self-similar viscous approximation~\eqref{e:ssva},~\eqref{e:cdd1}. For simplicity, we focus on the case when $B \equiv I$, but our analysis can be extended to  the more general case considered in~\cite{BiaSpi}. The main differences of our analysis from the analysis in~\cite{JoLeF:sv} are the following: first, we rely on completely different techniques since here we use center-stable manifold tools coming from the area of dynamical systems in the spirit of Bianchini and Bressan~\cite{BiaBrevv}. Our approach was inspired by the work of Dafermos~\cite[Section 9.8]{D} concerning the limit of the self-similar viscous approximation for the Riemann problem on the whole real line. However, the presence of the boundary amounts for additional challenges which are tackled by employing the Center-Stable Manifold Theorem. As a byproduct, our results apply directly to the analysis of non conservative systems, namely we handle the limit of the self-similar viscous approximation
\be
\label{e:nc}
          \partial_t U^{\varepsilon} + 
            A(U^{\varepsilon}) 
            \partial_x U^{\varepsilon}  =
            \varepsilon \,  t \, \partial_{xx} U^{\varepsilon}  
\ee
in the case when the $n \times n$ matrix $A (U)$ does not necessarily coincide with the Jacobian matrix $DF (U)$ of some flux function $F: \R^n \to \R^n$.  On the other side, the limit analysis in~\cite{JoLeF:sv} applies to both the non characteristic and (under genuine nonlinearity assumptions) to the boundary characteristic case, while here we restrict to the non characteristic case.

It should be noted that in the non conservative case the distributional solutions of the quasilinear hyperbolic system
$$
    \partial_t U + 
            A(U) 
            \partial_x U  = 0
$$ 
are not defined and we refer to Dal Maso, LeFloch and Murat~\cite{DalMasoLeFlochMurat} for possible definitions of weak solutions. We also quote Le Floch and Tzavaras~\cite{LeflochTzavaras} and the numerous references therein for the analysis of the Cauchy problems associated to the non conservative systems obtained by taking the limit $\varepsilon \to 0^+$ of~\eqref{e:nc}.
 
Before stating our main results, we introduce some notations. We consider system 
\be
\label{e:blnc}
\left\{
\begin{array}{ll}
             V'' = A(V) V' \\
             V(0) = U_b \qquad \lim_{y \to + \infty}  V(y) = \bar U 
\end{array}
\right.
\ee
and note that systems~\eqref{e:bl} and~\eqref{e:blnc} coincide in the case when $B \equiv I$ and $A( V) = DF( V)$. Next, we write the ordinary differential equation at the first line of~\eqref{e:blnc} as a first order system by setting
\be
\label{e:i:1o}
    \left\{
\begin{array}{ll}
           V' = W \\
            W' = A(V) W 
  \end{array}
\right.
\ee
and, by relying on the non characteristic condition~\eqref{e:i:sesp} and  the Stable Manifold Theorem (see e.g. Perko~\cite[Section 2.7]{Perko}), we conclude that there is a $k$-dimensional stable manifold ${\mathcal{M}^s (\bar U) \subseteq \R^n \times \R^n}$ which is invariant under~\eqref{e:i:1o} and consists of the data $(V_0, W_0)$ such that the solution of~\eqref{e:i:1o} satisfying $V(0) = V_0$ and $ W(0) = W_0$ has the following asymptotic behavior: 
$$
    \lim_{y \to + \infty}  | V(y) - \bar U|   = 0, \qquad   \lim_{y \to + \infty}  | W(y)| e^{cy/4} =0, 
$$
where $c$ is the same constant as in~\eqref{e:i:sesp}.
We then define the projection $\mathcal{M}^s_V (\bar U) \subseteq \R^n$ of $\mathcal{M}^s(\bar U)$ onto the first component as follows
\be
\label{e:stable}
   \mathcal{M}^s_V (\bar U) = \big\{  V \, \textrm{such that $(V,  W) \in \mathcal{M}^s (\bar U)$ for some $ W$} \big\} \;.
\ee
Now, we state the main result of the paper.

\begin{thm}
\label{t:main} Let condition \eqref{e:i:sesp} hold for some constant $c>0$ and a natural number $k$, $1\leq k \leq n-1$. Given a state $U_0\in\R^n$, there is a sufficiently small constant $\delta >0$ and a constant $\beta>0$ such that, for any 
$U_b$ and $\bar U$ satisfying $|U_b - U_0| \leq \beta \delta$ and $|U_b - \bar U| \leq \beta \delta$, we have that $U_b \in \mathcal M^s_V (\bar U)$ if and only if the following two properties hold. 
\begin{enumerate}
\item For every sequence $\{\bar U^\e \} \subseteq \R^n$ such that $\bar U^\e \to \bar U$ as $\e \to 0^+$, there exists a sequence of solutions $\Q^\e: [0, \delta ] \to \R^n$ to the problem
       \be
       \label{e:t:app}
        \left\{
        \begin{array}{lll}
               \displaystyle{  
               \e \frac{d^2 Q^\e}{ d \xi^2 } 
               =    \big[ A(\Q^\e) - \xi I \big]  \frac{d \Q^\e}{ d \xi}  }\\
                     \vspace{.2cm}
                \Q^\e (0)  = U_b \\           
                \Q^\e (\delta) = \bar U^\e \; ,   \\
        \end{array}
        \right.    
        \ee
        where $I$ denotes the $n \times n$ identity matrix. 
\item By setting $\Ve (\zz): = Q^\e ( \e\zz )$, the family $\{\Ve\}$ converges to a limit function 
$V^0: [0, + \infty[ \to \R^n$, as $\e \to 0^+$, uniformly on compact sets. The function $V^0$ solves the problem 
        \be
        \label{e:i:c2}
        \left\{
        \begin{array}{lll}
         \vspace{.2cm}
                 \dfrac{d^2 V^0}{d\zz^2}  =  A(V^0) \dfrac{dV^0}{d\zz} \\
                 \vspace{.2cm}
                V^0 (0) = U_b  \phantom{ \dfrac{d V^0}{d\zz} } \\  
                \vspace{.2cm}
                \displaystyle\lim_{\zz \to + \infty} V^0(\zz) = \bar U \phantom{ \dfrac{d V^0}{d\zz} }
                \\
                \displaystyle\lim_{\zz \to + \infty} \left|  \dfrac{d V^0}{d\zz} \right| e^{c \zz/4} =0.
        \end{array}
        \right.    
        \ee        
\end{enumerate}
\end{thm}

Before we proceed, we make some remarks. First, it is obvious that the solutions $\Q^\e(\xi)$ to the ordinary differential equations~\eqref{e:t:app} yield the self-similar solutions $U^\e(t,x)$ of~\eqref{e:nc} by writing $U^\e(t,x)=\Q^\e(x/t)$. Second, by relying on the Stable Manifold Theorem, it follows immediately that if system ~\eqref{e:i:c2} admits a solution, then $U_b \in \mathcal M^s_V (\bar U)$. However, the opposite direction of the implication in the statement of the theorem is not trivial. In other words, the proof is mainly devoted to showing that from $U_b \in \mathcal M^s_V (\bar U)$, then properties (1) and (2) of the theorem follow. Third, Theorem~\ref{t:main} can be reformulated by stating that the boundary layers generated by the self-similar viscous approximation are the same as the boundary layers~\eqref{e:blnc} generated by the classical vanishing viscosity approximation. Below, we discuss an application of this fact concerning the boundary Riemann problem and the corresponding precise result is stated in Section~\ref{s:cssva} as Proposition~\ref{t:cor}.

Consider the self-similar vanishing viscosity approximation 
\begin{equation}
\label{e:i:ssvva}
\left\{
\begin{array}{lll}
            \partial_t U^\e + A(U^\e) \partial_x U^\e = \e t  \partial^2_{xx} U^\e \\
            U^\e (t, 0) = U_b \\
            U^\e (0, x ) = U_0 \\
\end{array}
\right.
\end{equation} 
and the classical vanishing viscosity approximation
\begin{equation}
\label{e:i:vva}
\left\{
\begin{array}{lll}
            \partial_t Z^\e + A(Z^\e) \partial_x Z^\e = \e \partial^2_{xx} Z^\e \\
            Z^\e (t, 0) = U_b \\
            Z^\e (0, x ) = U_0 \\
\end{array}
\right.
\end{equation} 
of a Riemann problem. Under the assumption that $|U_b - U_0|$ is sufficiently small, convergence results for~\eqref{e:i:ssvva} and~\eqref{e:i:vva} have been established by Joseph and LeFloch~\cite{JoLeF:sv, JoLeF:nc} and by Ancona and Bianchini~\cite{AnBia}, respectively. 
 
By combining the analysis in Dafermos~\cite[Section 9.8]{D} and Theorem~\ref{t:main}, we infer that, under appropriate assumptions (for completeness we give the precise result in Proposition~\ref{t:cor} of Section~\ref{s:comp}), the limits of~\eqref{e:i:ssvva} and~\eqref{e:i:vva} coincide. An analogous result has been proved for conservative systems in Christoforou and Spinolo~\cite{ChristoforouSpinolo} in both the characteristic and the non characteristic boundary case. Here we extend this analysis to non conservative systems but we focus on the non characteristic case. The reason why these results are not \emph{a priori} obvious is that in general, for initial-boudary value problems, the limit of a viscous approximation~\eqref{e:va},~\eqref{e:cdd} depends on the choice of the viscosity matrix $B$ and hence, on the type of viscous approximation. 

The exposition of the paper is organized as follows: in Section~\ref{s:cssva} we compare the limits of the classical vanishing viscosity and self-similar  
viscous approximation of a boundary Riemann problem. More precisely, in Subsection~\ref{s:pre}, we present existing results concerning the Riemann problem established by Dafermos~\cite{D}  for the self-similar viscous approximation and by Bianchini and Bressan~\cite{BiaBrevv} and Ancona and Bianchini~\cite{AnBia}  for the vanishing viscosity approximation. Using these results and applying Theorem~\ref{t:main}, in Subsection~\ref{s:comp} we show that, under reasonable assumptions, the limits of~\eqref{e:i:ssvva} and~\eqref{e:i:vva} coincide. The proof of Theorem~\ref{t:main} is provided in Section~\ref{s:proof}.

\section{On the classical and the self-similar viscous approximation of a boundary Riemann problem}
\label{s:cssva}
The aim of this section is to compare the limits of the self-similar~\eqref{e:i:ssvva} and the classical vanishing 
viscosity~\eqref{e:i:vva} approximation of a boundary Riemann problem. We first quote existing results related to the wave fan curves corresponding to these two different approximations and then we show in Proposition~\ref{t:cor} that the two limits are the same.

\subsection{Preliminary results: construction of the wave fan curves}\label{s:pre} 
In~\cite{lax}, Lax studied self-similar, distributional solutions of the system of conservation laws~\eqref{e:cl} by imposing some technical assumptions on the structure of the flux function $F$. In particular, for every $i=1, \dots, n$, he constructed  the \emph{$i$-wave fan curve} $T^i (s_i, U^+)$ passing through a given state $U^+\in\R^n$ with $s_i\in\R$ being the variable parameterizing the curve. The curve $T^i (s_i, U^+)$ takes values in $\R^n$ as $s_i$ varies in  a neighborhood of $0$ and for fixed $s_i$ small enough, the Riemann problem 
\be
\label{e:rie}
\left\{
\begin{array}{ll}
           \partial_t U + \partial_x \big[ F(U) \big] =0, \\
            U(0, x) =
            \left\{
            \begin{array}{ll}
                  U^+ & x >0 \\
                  T^i (s_i, U^+) & x < 0 \\ 
            \end{array}            
            \right.
\end{array}    
\right.
\ee 
admits a self-similar solution which is either a rarefaction, or a single contact discontinuity or a single shock satisfying the Lax admissibility condition. This condition was introduced in the same paper~\cite{lax}. Lax's construction was later extended to very general cases in the works by Liu~\cite{Liu:rie, Liu:adm}, Tzavaras~\cite{Tz} and Bianchini~\cite{Bia:riemann}. 

Also, Bianchini and Bressan~\cite{BiaBrevv} and Dafermos~\cite[Section 9.8]{D} described the \emph{$i$-wave fan curve} obtained by taking the limit $\varepsilon \to 0^+$ in the classical vanishing viscosity
\be\label{e:clas}
        \partial_t  Z^{\e} + A(Z^\e) \partial_x Z^{\e} = \e \partial_{xx}^2 Z^{\e}\;
\ee
and in the self-similar viscous approximation~\eqref{e:nc}, respectively. We go over the constructions in~\cite{BiaBrevv} and~\cite[Section 9.8]{D} in Subsections~\ref{s:wavefanvan} and~\ref{s:daf},  respectively, while in Subsection~\ref{s:AnBia} we quote the description of the limit of~\eqref{e:i:vva} in the boundary case provided by Ancona and Bianchini~\cite{AnBia}. We note that the constructions in~\cite{BiaBrevv, D, AnBia} include the case of non conservative systems.       
 
\subsubsection{Wave fan curves induced by the classical vanishing viscosity approximation}
\label{s:wavefanvan}
Here, we describe the \emph{$i$-wave fan curve} $T^i (s_i, U^+)$ as constructed in~\cite[Section 14]{BiaBrevv}. 

Given a matrix $A(U)$ satisfying strict hyperbolicity~\eqref{sh}, we denote by $R_1(U)$, $\dots$, $R_n(U)$ a basis of right eigenvectors.  

We first assume $s_i >0$ and then, consider the following fixed problem: 
\be \label{e:biabre}
        \left\{
	\begin{array}{ll}
	V_i (\tau)  &  = \displaystyle U^+ + \int_0^{\tau} \hat R_i \big(V_i(s) , \, \omega_i (s) , \xi_i(s)  \big)  \, ds  \\ \\
	\omega_i (\tau)    & = \hat{f}(\tau) - \hat{g}(\tau) \\ \\
	\xi_i (\tau)  & = \displaystyle \frac{d \hat{g} }{d \tau}  (\tau), \\ \\
	\end{array}
	\right. 
\ee
where we used the notation 
$$
    \hat  f(\tau) \doteq  \int_0^{\tau} \hat \lambda_i \big(V_i(s) , \, \omega_i (s) , \xi_i(s) \big)  \, ds,
$$
while $\hat g$ denotes the concave envelope of $\hat f$ on the interval $[0, \, s_i]$. Also, the so-called {generalized eigenvectors} $\hat R_i$ and \emph{generalized eigenvalues} $\hat \lambda_i$ are defined in~\cite[Section 14]{BiaBrevv} by considering the travelling waves $V$ to~\eqref{e:clas}, namely solutions of the system
\be \label{e:biabre2}
         \left\{
	\begin{array}{ll}
	V'& = W \\ 
	W' & = \Big[ A(V) - \xi I \Big] W \\ 
	\xi' & = 0. \\ 
	\end{array}
	\right. 
\ee
By restricting system~\eqref{e:biabre2} on a center manifold about the equilibrum point $V = U^+$, $W = \vec 0$, $\xi = \lambda_i (U^+) $, one eventually determines the system 
\be \label{e:biabre3}
         \left\{
	\begin{array}{ll}
	V'& =  \hat R_i (V, \, \omega_i, \xi) \omega_i  \phantom{\Big(}  \\
	\omega_i ' & = \big( \hat \lambda_i (V, \, \omega_i, \xi)  - \xi \big) \omega_i  \phantom{\Big(} \\ 
	\xi' & =0, \phantom{\Big(}  \\
	\end{array}
	\right. 
\ee 
where $\omega_i$ is a scalar unknown and the generalized eigenvector $\hat R_i$ is a suitable function taking values in $\R^n$ and satisfying $ \hat{R}_i (U^+, \, 0, \lambda_i (U^+)) = R_i (U^+)$. 

By relying on the Contraction Map Theorem, one can show that problem~\eqref{e:biabre} has a unique solution $(V_i, \omega_i, \xi_i)$ belonging to a suitable metric space and then set $T^i (s_i, U^+) = V_i (s_i)$. The construction in the case when $s_i <0$ follows similarly by taking $\hat{g}$ in~\eqref{e:biabre} to be the convex envelope of $\hat{f}$ instead of the concave one.

\subsubsection{Vanishing viscosity limit of a boundary Riemann problem}
\label{s:AnBia}
In~\cite{AnBia}, Ancona and Bianchini showed that, if $|U_0- U_b|$ is sufficiently small and the non characteristic condition~\eqref{e:i:sesp} holds, then the limit $\varepsilon \to 0^+$ of the vanishing viscosity approximation~\eqref{e:i:vva} is determined by imposing the condition
\begin{equation}
\label{e:fpt2}
         U_b = \mathcal G_{U_0} ( S, s_{k+1}, \dots s_n) \doteq  \phi_s \Big( S, T^{k+1} \big( s_{k+1}, \dots, T^n (s_n, U_0)  \dots \big) \Big)\;.
\end{equation}
Here, $ S \in \R^k$ and $\phi_s$ is a map parameterizing the stable manifold of~\eqref{e:i:1o} about the equilibrium point $(\bar U, \vec 0)$, i.e. the manifold $  \mathcal M^s_V (\bar U)$ in~\eqref{e:stable} can be expressed as follows:
$$
    \mathcal M^s_V (\bar U) = \{ \phi_s ( S, \bar U)\, \textrm{ for some vector 
    $ S \in \R^k$} \}. 
$$
Also, $T^i(s_i, \cdot)$ is the \emph{curve of admissible states} defined by Bianchini and Bressan in~\cite{BiaBrevv} whose construction is sketched in Subsection~\ref{s:wavefanvan}.

It can be shown that~\eqref{e:fpt2} uniquely determines the values of $S$, $s_{k+1}, \dots, s_n$ and that the solution $Z(t,x)$ obtained as the $\e\to 0^+$ limit of the classical vanishing viscosity approximation~\eqref{e:i:ssvva} admits the following representation: 
\begin{equation}
\label{e:z}
         Z(t, x) = 
          \left\{
    \begin{array}{ll}
    		     \bar U       &         \text{if}\quad  x/t < \xi_{k+1} (s_{k+1})   \\
		   V_j (\tau)   &         \text{if}\quad  x/t= \xi_j (\tau),\quad \text{for}\; \tau\in]0,s_j[\;\text{ or } ]s_j,0[,\; j=k+1,\dots,n       \\
               V_j (s_j)     &         \text{if}\quad    \xi_{j} (0) < x / t < \xi_{j+1}(s_{j+1}) \quad \text{for}  \; j=k+1,\dots,n-1 \\     
               U_0           &          \text{if}\quad   x/t > \xi_n (0),                \\
    \end{array}
    \right.
\end{equation} 
where $(s_{k+1}, \dots, s_n)$ are determined by using~\eqref{e:fpt2}, $V_j$ and $\xi_j$ are the solutions of the fixed point problem~\eqref{e:biabre} and the trace $\bar U$ is given by
$$
\bar U =T^{k+1} \big( s_{k+1}, \dots ,T^n (s_n, U_0)  \dots \big)\;.
$$

\subsubsection{Wave fan curves induced by the self-similar viscous approximation}
\label{s:daf} 
Now, we present the construction of the wave fan curves induced by the self-similar viscous approximation as was constructed by Dafermos in~\cite[Section 9.8]{D}

To construct the $i$-wave fan curve $\phi_i(s_i,  U^+)$  that emanates from some state $U^+$, we consider the system
\be
\label{e:visode2}
            \displaystyle{  
               \e \frac{d^2 Q^\e}{ d \xi^2 } 
               =    \big[ A(Q^\e) -\xi I \big]  \frac{d Q^\e}{ d \xi}  }
\ee
satisfied by the self-similar solutions and we decompose the first derivative along the right eigenvectors of $A(U)$:
\be \label{e:dec}
\frac{dQ^\e}{d\xi} =\sum_{j=1}^n a_j(\xi) R_j(Q^\e(\xi)).
\ee
By taking a dual basis of left eigenvectors $L_1(U), \dots, L_n(U)$, the component $a_j$ is given by $a_j = \langle L_j, Q^\e \rangle$, $j=1,\dots,n$, where $\langle \cdot, \cdot \rangle$ denotes the standard scalar product.

Substituting~\eqref{e:dec} into~\eqref{e:visode}, we arrive at
$$
     \e \frac{d{a_j}}{d\xi}   =  [\lambda_j(V)-\xi] a_j + \e  \sum_{h,l=1}^n \beta_{jhl} a_h \, a_l, 
$$
where the coefficients $\beta_{jhl}$ are given by
\be
\label{e:beta}
     \beta_{jhl} (Q^\e) = -  \langle D L_j (Q^\e) R_h (Q^\e) , R_l (Q^\e) \rangle 
\ee
and $D L_j$ denotes the Jacobian matrix of the field $L_j$. 
We now perform a change of variables by setting $\xi=\e\zeta$  and rescale the components $a_j$ by setting  $\omega_j : = \e a_j$ for every $j=1, \dots n$. The self-similar viscous approximation $Q^\e$ to ~\eqref{e:visode2} is denoted by $V^\e$ in the new variable $\zz$, i.e. $V^\e(\zz)\doteq Q^\e(\e\zz)$. However, for the convenience of the reader, we abuse the notation, for now, and we drop the index $\e$ from $V^\e$ since the $\e$ dependence is clear in what follows. In Subsection~\ref{sss:app}, we will return to the original notation $V^\e$. Using decomposition~\eqref{e:dec}, we rewrite system~\eqref{e:visode2} as an autonomous first order system 
\be
	\label{S3: autonomous ODE}
	\left\{
	\begin{array}{ll}
	V'& = \displaystyle\sum_{j=1}^n \omega_j R_j(V) \\ 
	\omega_j' & =[\lambda_j(V)-\xi]\omega_j+\displaystyle\sum_{h,l=1}^n\beta_{jhl} \omega_h \,\omega_l    \quad 
	\textrm{for every $j = 1, \dots, n$} \\ 
	\xi' & =\e\\ 
	\e' & = 0\;, \phantom{\displaystyle \sum_{i=1}^n}
	\end{array}
	\right. 
\ee
where $'$ denotes differentiation with respect to $\zeta$. By linearizing~\eqref{S3: autonomous ODE} about the equilibium point 
$$
    {V = U^+,\,\,\, \omega_1 =0, \dots ,\,\omega_n =0,\,\,\, \xi =\lambda_i( U^+) ,\,\,\,\e =0}, 
$$ we get the linear system
\be \label{e:cen:lin}
         \left\{
	\begin{array}{ll}
	V'& = \displaystyle\sum_{j=1}^n \omega_j R_j( U^+) \\ 
	\omega_j' & =[\lambda_j(U^+)- \lambda_i (U^+)]    \omega_j \quad 
	\textrm{for every $j = 1, \dots, n$} \phantom{ \displaystyle{\sum}} \\ 
	\xi' & =\e \phantom{ \displaystyle{\sum_{j=1}^n}}\\ 
	\e' & = 0 \;. \phantom{ \displaystyle{\sum}}  
	\end{array}
	\right. 
\ee
The center subspace $\mathcal N_i$ of system~\eqref{e:cen:lin} consists of the vectors $(V,\omega,\xi, \e)\in\mathbb{R}^{2n+2}$, with $\omega_j=0$ for $j\ne i$, thus,  $\mathcal N_i$ has dimension $n+3$. By the Center Manifold Theorem, there is an $(n+3)$-dimensional manifold containing all the solutions of~\eqref{S3: autonomous ODE} that sojourn in a small enough neighbourhood of the equilibrium point $(U^+, \vec 0, \lambda_i (U^+), 0)$. Such a center manifold is parameterized by $\mathcal N_i$ and it is tangent to the center subspace at the equilibrium. 
In general, the center manifold of~\eqref{S3: autonomous ODE} about $(U^+, \vec 0, \lambda_i (U^+), 0)$ is not unique and therefore, in the following, we fix one and we denote it by $\mathcal M_i$.  We refer the reader to the notes by Bressan in~\cite{Bre:notes} for an extended discussion about the Center Manifold Theorem.

As it is shown in~\cite[Section 9.8]{D}, system~\eqref{S3: autonomous ODE} on $\mathcal M_i$ is equivalent to  
\be \label{e:cen:red}
         \left\{
	\begin{array}{ll}
	V'& =   R_i^\sharp (V, \, \omega_i, \xi, \e) \omega_i  \phantom{\Big(}  \\
	\omega_i ' & =\big(  \lambda^\sharp _i (V, \, \omega_i, \xi, \e) - \xi \big)  \omega_i  \phantom{\Big(} \\ 
	\xi' & =\e \phantom{\Big(}  \\
	\e' & = 0\;,\phantom{\Big(} \\
	\end{array}
	\right. 
\ee 
where the functions $ R_i^\sharp $ and $ \lambda_i^\sharp $ take values in $\R^n$ and $\R$, respectively,
and satisfy
\be \label{e:con:ri}
       R_i^\sharp  (U^+, \, 0, \lambda_i(U^+), 0) = R_i (U^+) \quad \textrm{and} \quad  \lambda_i^\sharp  (U^+, \, 0, \lambda_i(U^+), 0) = \lambda_i (U^+).  
\ee     
It should be noted that system~\eqref{e:cen:red} consists of $(n+3)$ equations.

To define the value $\phi_i (s_i, U^+)$ attained by the $i$-wave fan curve of admissible states emanating from $U^+$ at $s=s_i >0$, we consider  the fixed point problem
\be \label{e:cen:fi}
        \left\{
	\begin{array}{ll}
	V_i (\tau)  &  = \displaystyle U^+ + \int_0^{\tau}  R_i^\sharp  \big(V_i(s) , \, \omega_i (s) , \xi_i(s), 0 \big)  \, ds  \\ 
	\omega_i (\tau)    & = f^\sharp (\tau) - g^\sharp (\tau) \\
	\xi_i (\tau)  & = \displaystyle \frac{d g^\sharp  }{d \tau}  (\tau) \\ 
	\end{array}
	\right. 
\ee
over the interval $[0, s_i]$. Here,
$$
    f^\sharp  (\tau) \doteq  \int_0^{\tau}  \lambda_i^\sharp  \big(V_i(s) , \, \omega_i (s) , \xi_i(s), 0 \big)  \, ds,
$$
while $g^\sharp $ denotes the concave envelope of $f^\sharp $ on the interval $[0, \, s_i]$. By relying on a fixed point argument, it can be shown that~\eqref{e:cen:fi} admits a unique solution $\big(V_i(\tau), \omega_i(\tau), \xi_i(\tau) \big)$ over $[0, s_i]$. Then, the $i$--wave fan curve of admissible states is defined by setting  $\phi_i({s_i}, U^+):= V_i(s_i)$. 

If $s_i <0$, the construction is similar, except that the function $g$ in~\eqref{e:cen:fi} is the \emph{convex} envelope of $f$ over the interval $[s_i, 0]$. Moreover, it can be shown  that, by construction, the curve $\phi_i (s_i, U^+)$ is tangent  to the eigenvector $R_i(U^+)$ at $s_i =0$.          

\subsection{Comparison between the self-similar and the classical vanishing viscosity approximation}
\label{s:comp}

Here, we prove that under appropriate conditions on the limit obtained by the self-similar viscous approximation~\eqref{e:i:ssvva}, the solution to the boundary Riemann problem obtained via the classical vanishing viscosity approximation~ \eqref{e:i:vva} coincides with the one obtained via the self-similar viscous approximation~\eqref{e:i:ssvva}. This is not a priori obvious since it is known that, in general, for initial-boundary value problems the limit depends on the type of viscous approximation.

The comparison between the two limits is established using Theorem~\ref{t:main} and the construction of the wave fan curves described in Subsection~\ref{s:pre}. The following proposition states the result:

\begin{prop}
\label{t:cor}
         Let $U_0\in\R^n$ and assume that condition~\eqref{e:i:sesp} holds for an appropriate constant $c>0$ and for some natural number $k$, $1\le k\le n-1$.  Then, there exists a sufficiently small constant $\delta >0$ such that given $U_b$ with $|U_0 -  U_b| \leq \delta$, the family of self-similar viscous approximations $U^{\e}$ satisfying~\eqref{e:i:ssvva} converges, up to subsequences, to a self-similar limit function $U(t, x)=\Q (x/t)$. Assume that the following two conditions are both satisfied:
                  \begin{enumerate}
                  \item[(a)] let $\bar U$ denote the trace
                  $$
                      \bar U \dot{=} \lim_{\xi \to 0^+} \Q(\xi),
                  $$
                  then $\bar U$ satisfies conditions (1) and (2) in the statement of Theorem~\ref{t:main}.
                 \item[(b)] The function $U$ admits the following representation:
                 \begin{equation}
\label{e:u}
         U(t,x) = 
          \left\{
    \begin{array}{ll}
    		     \bar U       &         \text{if $ x/t < \xi_{k+1} (s_{k+1})$}   \\
		   V_j (\tau)   &         \text{if $x/t = \xi_j (\tau)$}    \quad \text{for}\; \tau\in]0,s_j[\;\text{ or } ]s_j,0[,\; j=k+1,\dots,n   \\
               V_j (s_j)     &         \text{if $ \xi_{j } (0) < x/t < \xi_{j+1}(s_{j+1}) $} \quad \text{for}  \; j=k+1,\dots,n-1 \\     
               U_0           &           \text{if  $x/t > \xi_n (0),$}                \\
    \end{array}
    \right.
\end{equation}
for suitable small values $s_{k+1}, \dots, s_n\in\R$, where the functions $V_j$ and $\xi_j$ are solutions of the fixed point problem~\eqref{e:cen:fi}.
         \end{enumerate}
         Then the limit $U(t,x)$  of the self-similar viscous approximation $U^\e$ defined by~\eqref{e:i:ssvva} coincides with the unique limit $Z(t,x)$ of the classical vanishing viscosity approximation $Z^\e$ defined by~\eqref{e:i:vva}.
          \end{prop} 

\begin{proof}
The convergence result follows from the analysis in Joseph and LeFloch~\cite{JoLeF:sv}. Indeed, the proof of the compactness result in the first part of~\cite{JoLeF:sv} can be straightforwardly extended to the non conservative case.  We now focus on the characterization of the limit. 

Theorem~\ref{t:main} states that condition ($a$) in the statement of Proposition~\ref{t:cor} is equivalent to assuming that 
$U_b \in \mathcal M^s_V (\bar U)$ and hence by relying on conditions ($a$) and ($b$), we deduce the relation
\begin{equation}
\label{e:fpt}
         U_b = \mathcal F_{U_0} (S, s_{k+1}, \dots s_n) =  \phi_s \Big(S, \phi_{k+1} \big( s_{k+1}, \dots \phi_n (s_n, U_0)  \dots \big) \Big),
\end{equation}
where $\phi_s$ is as in~\eqref{e:fpt2} a function parameterizing the stable manifold and the function $\phi_j$, $j=k+1, \dots, n$ is the $j$--wave fan curve induced by the self-similar viscous approximation and is defined by~\eqref{e:cen:fi}. We note that $\mathcal F_{U_0}$ is a function from $\R^n$ to $\R^n$ and we claim that it is locally invertible. Indeed, we recall that the stable manifold $\mathcal M^s (\bar U)$ is tangent at the origin to the stable space $V^s(\bar U)$ and the curve $\phi_j (s_j, U^+)$ is tangent at $s_j =0$ to the eigenvector $R_j (U^+)$ and therefore, the Jacobian matrix of $ \mathcal F_{U_0}$ evaluated at the point $S = \vec 0$, $s_{k+1} = \dots = s_n=0$ is the $n \times n$ matrix whose columns are the eigenvectors $R_1 (U_0), \dots , R_n ( U_0)$. By relying on the Local Invertibility Theorem, we conclude that, if $|U_0 - U_b|$ is sufficiently small, then relation~\eqref{e:fpt} uniquely determines the values of $S, s_{k+1}, \dots, s_n$.

We can now conclude the proof  by combining the following two remarks. First, it should be noted that, in general, the center manifold about an equilibrium point is not unique and, hence, the functions $\hat R_i$ and $\hat \lambda_i$ in~\eqref{e:biabre}, \eqref{e:biabre3} \emph{depend} on the choice of the center manifold. However, as proven in~\cite[Remark 14.2, page 305]{BiaBrevv}, the limit $Z(t, x)$ that admits representation~\eqref{e:z} \emph{does not} depend on this choice. Second, by construction, the solutions of system~\eqref{e:cen:red} satisfying the additional condition $\e=0$ are solutions of~\eqref{e:biabre2} lying on a center manifold about the equilibrium point $\big(V_i =  U^+, W = \vec 0, \xi = \lambda_i (U^+) \big)$. By comparing~\eqref{e:cen:fi} and~\eqref{e:biabre} and using the fact that the fixed point of~\eqref{e:biabre} does not depend on the choice of the center manifold, we conclude that the function $U$ in~\eqref{e:u} coincides with the unique function $Z$ in~\eqref{e:z}. The proof is complete.
\end{proof}

\begin{remark} It should be noted that Joseph and LeFloch in~\cite[Theorem 4.2]{JoLeF:sv} showed that the limit $U$ of the self-similar approximation~\eqref{e:i:ssvva} satisfies condition ($a$) of Proposition~\ref{t:cor} for the conservative case, $A(U) = DF(U)$.
\end{remark}

\section{Proof of Theorem~\ref{t:main}}
\label{s:proof}
In this section, we provide the proof of Theorem~\ref{t:main}. 

By relying on the proof of the Stable Manifold Theorem (see Perko~\cite[Section 2.7]{Perko}) we have that, if there exists a function $V^0$ satisfying~\eqref{e:i:c2}, then $U_b \in \mathcal M^s_V (\bar U).$
Establishing the opposite implication amounts to show that, if $U_b \in \mathcal M^s_V (\bar U)$, then properties (1) and (2) of the theorem hold. To accomplish this, we proceed in four  steps: 
\begin{enumerate}
\item[$\bullet$] in Section~\ref{sus:cs}, we restrict system ~\eqref{S3: autonomous ODE}  to a center-stable manifold;
\item[$\bullet$] in Section~\ref{sss:app}, we construct the approximating sequence;
\item[$\bullet$] in Section~\ref{sus:limit}, we establish the convergence of the sequence and study the limit;
\item[$\bullet$] in Section~\ref{sus:conclu}, we eventually conclude the proof of Theorem~\ref{t:main}.
\end{enumerate}
Note that in the following we fix a state $\bar U \in \R^n$ and we exhibit constants $c_V$ and $\delta$ such that, for every $U_b$ satisfying $|\bar U - U_b | \leq c_V \delta$ and $U_b \in \mathcal M^s_V (\bar U)$, properties (1) and (2) in the statement of Theorem~\ref{t:main} are satisfied. The values of $c_V$ and $\delta$ \emph{a priori} depend on the state $\bar U$, however,  as a matter of fact, the values of these constant only depend on some regularity coefficients of suitable functions constructed in Section~\ref{sus:cs}. It turns out that the values of these functions are uniformly bounded on sufficiently small compact sets and hence, the statement of Theorem~\ref{t:main} is valid for every $\bar U$ in a small enough neighborhood of a given state $U_0$.

\subsection{Solutions lying on a center-stable manifold}
\label{sus:cs}
In this subsection, we analyze the solutions lying on a center-stable manifold for the boundary layers.

By linearizing system~\eqref{S3: autonomous ODE} about the equilibrium point ${(V,\omega,\xi,\e)= (\bar U, \vec 0, 0 ,0)}$, we obtain
\be\label{e:censt:li}
         \left\{
	\begin{array}{ll}
	V'& = \displaystyle\sum_{j=1}^n \omega_j R_j(\bar U) \\ 
	\omega_j' & = \lambda_j(\bar{U}) \omega_j \quad \text{for every $j =1, \dots, n$}\\ 
	\xi' & =\e \phantom{\displaystyle \sum_{j=1}^n} \\ 
	\e' & = 0 \phantom{\displaystyle \sum}\;.
	\end{array}
	\right. 
\ee
By the non characteristic condition~\eqref{e:i:sesp}, we have that $ \lambda_j(\bar{U}) \neq 0$ for every $j=1, \dots, n$ and therefore, the center-stable space to~\eqref{e:censt:li} is described by the vectors $(V, \omega_1, \dots \omega_n, \xi, \e)$ satisfying 
$V \in \R^n$ and  ${\omega_{k+1} =0, \dots ,\omega_n=0}$. This implies that its dimension is $n + k+2$. In the following, to simplify the exposition we use the notation 
$
   W = (\omega_1, \dots ,\omega_n)^t.
$

Next, we apply the Center-Stable Manifold Theorem and we refer the reader to the book by Katok and Hasselblatt~\cite{KHass} for an extended discussion on this subject. Here, we only recall the properties needed in this article. By the Center-Stable Manifold Theorem, it follows that in a sufficiently small neighbourhood of the equilibrium point $(\bar U, \vec 0, 0 , 0)$ there is a so-called center-stable manifold, which is locally invariant for system~\eqref{S3: autonomous ODE} and it has dimension $n + k+2$ as the center-stable space. 
In general, the center-stable manifold of a given equilibrium point is not unique. However, any center-stable manifold of the equilibrium $(\bar U, \vec 0, 0 , 0)$ for~\eqref{S3: autonomous ODE} is always tangent to the center-stable space at $(\bar U, \vec 0, 0 , 0)$. Moreover, let $c$ be the same constant as in~\eqref{e:i:sesp}, then any center-stable manifold contains all the orbits 
$\big(V(\zeta), W (\zeta), \xi (\zeta), \e(\zeta)\big)$ that are confined in a sufficiently small neighborhood of $(\bar U, \vec 0, 0 , 0)$ and satisfy
$$
    \lim_{\zeta \to + \infty} \Big[ |V(\zeta) - \bar U | + 
    | W (\zeta)|+ | \xi (\zeta)| + |\e(\zeta)| \Big] \exp (- c \zeta / 2) =0 \;.
$$
By taking the value of $\delta$ in the statement of Theorem~\ref{t:main} sufficiently small, we now fix a center-stable manifold $\mathcal M^{cs}$ defined in the neighborhood of $(\bar U, \vec 0, 0, 0)$ given by
\begin{equation}
\label{e:bds}
       |V - \bar U| \leq \delta, \quad |W | \leq \delta, \quad  |\xi| \leq \delta, \quad |\e| \leq \delta. 
\end{equation}       
Note that actually $\mathcal M^{cs}$ depends on $\bar U$, but to simplify notations we do not explicitly indicate this dependence.  
Also, using the notation ${W_{cs} = ( \omega_1, \dots \omega_{k})^t}$, the parameterization of $\mathcal M^{cs}$ can be chosen in such a way that 
\begin{equation}
\label{e:w}
     \mathcal M^{cs} = \big\{ (V,W, \xi, \e) \; \textrm{such that} \; W = \Phi (V,W_{cs}, \xi, \e) \big\}
\ee
for a suitable function $\Phi$. 
In particular, if we restrict system~\eqref{S3: autonomous ODE} to $\mathcal M^{cs}$, then our unknowns are the functions $V(\zeta)$, $\wcs(\zeta)$, $\xi(\zeta)$, $\e(\zeta)$.  In the following lemma, we derive the ODE satisfied by $(V, \wcs, \xi, \e)$.

\begin{lemma}
\label{l:redu} Let $\mathcal M^{cs}$ be the aforementioned center-stable manifold of~\eqref{S3: autonomous ODE}.  If the constant $\delta$ is sufficiently small, then any solution of~\eqref{S3: autonomous ODE} lying on $\mathcal M^{cs}$ and satisfying bounds~\eqref{e:bds} 
                     is also a solution of the system
          \be
          \label{e:redu}
         \left\{
	\begin{array}{ll}
	V'& = \Psi \, \wcs   \phantom{\Big( }  \\ 
	\wcs' & =  \Big( \Lambda - \xi \Upsilon \Big) \wcs   \\ 
	\xi' & =\e \phantom{\Big( }  \\ 
	\e' & = 0  \phantom{\Big( }\;.
	\end{array}
	\right. 
\ee         
Here, the function $\Psi=\Psi ( V, \wcs, \xi, \e)$ takes values in the space of ${n \times k}$ matrices and both the functions  
$\Lambda= \Lambda ( V, \wcs, \xi, \e)$ and $\Upsilon=\Upsilon ( V, \wcs, \xi, \e)$ take values in the space of ${k \times k}$ matrices. Moreover, $\Lambda{(\bar U, \vec 0, 0, 0)}$ is the diagonal matrix with eigenvalues 
${\lambda_1(\bar U), \dots \lambda_{k} (\bar U)}$, while $\Upsilon{(\bar U, \vec 0, 0, 0)}$ is the identity matrix. 
\end{lemma} 
\begin{proof} 
First, we observe that the vector $(V, W = \vec 0, \xi, \e)$ belongs to $\mathcal M^{cs}$ for any $V$, $\xi$ and $\e$. Hence,
$$
     \Phi (V, \wcs = \vec 0, \xi, \e) = \vec 0 \quad \textrm{for all $V$, $\xi$, $\e$},
$$
where $\Phi$ is the function given in~\eqref{e:w}. By relying on the regularity of the map $\Phi$ and using~\eqref{e:w}, we get the expression
\be\label{e:wcs}
 W=  \Phi (V, W_{cs}, \xi, \e) = \Phi_{cs}  (V, W_{cs}, \xi, \e) \wcs
\ee
for a suitable regular map $\pcs$ taking values in the space of the $n \times k$ matrices. 
Since the manifold $\mathcal M^{cs}$ is tangent to the center-stable space at the equilibrium ${(V, \wcs, \xi, \e)=(\bar U, \vec 0, 0 ,0)}$, then the 
columns of the matrix 
${ \Phi_{cs}  (\bar U,  \vec 0, 0 ,0)}$ are the eigenvectors $R_1 (\bar U), \dots R_k (\bar U)$:
\begin{equation}
\label{e:columns}
           \Phi_{cs}  ( \bar U,  \vec 0, 0 ,0) = \Big( R_1(\bar U) \vert \dots \vert R_{k} (\bar U) \Big).
\end{equation}
It should be noted that the first equation of system~\eqref{S3: autonomous ODE} can be written as
$$
     V' =  \Big( R_1(V) \vert \dots \vert R_{n} (V) \Big) W
$$
and thus, by substituting~\eqref{e:wcs} above, we get that 
\begin{equation}
\label{e:b}
    V' = \Psi( V, \wcs, \xi, \e) \wcs,
\ee
for a suitable matrix $\Psi$ taking values in the space of the $n \times k$ matrices. 

Next, we differentiate~\eqref{e:wcs} to get 
$$
    W'= \pcs' \wcs + \pcs \wcs'. 
$$
We also have 
\begin{equation}
\label{e:der}
\begin{split}
             \pcs'  ( V, \wcs, \xi, \e)  \wcs   &
           =   B V' +  C  \wcs' + \xi'   \partial_{\xi} \pcs   \wcs  = \\
           &   = B    \Psi   \wcs +  C  \wcs ' + \e ( \partial_{\xi}  \pcs  ) \wcs\;,   \\
\end{split}
\end{equation}
where the functions $B=B (V, \wcs, \xi, \e)$ and $C=C ( V, \wcs, \xi, \e)$ take values in the space of $n \times n$ and $n \times k$ matrices, respectively and satisfy
\be
\label{e:B}
            B ( V, \vec 0, \xi, \e) = \mathbf{0}_{n \times n} \quad  C ( V, \vec 0, \xi, \e)  = \mathbf{0}_{n \times k}
             \quad \textrm{ for every $V$, $\xi$ and $\e$.}
\ee   
The exact expressions of $B$ and $C$ are not relevant here. 

By substituting~\eqref{e:der} into the second equation of~\eqref{S3: autonomous ODE}, we obtain
\begin{equation}
\label{e:der2}
\begin{split}
            W' 
            & =  \pcs \wcs '  + B    \Psi   \wcs +  C  \wcs ' + \e (  \partial_{\xi}  \pcs )  \wcs = \\
            & = \Big( \mathrm{Diag}\big( \lambda_1 , \dots ,\lambda_n \big)- \xi I \Big) \pcs \wcs + F\wcs, \\
\end{split}
\end{equation}
where the $\mathrm{Diag}\big( \lambda_1 , \dots ,\lambda_n  \big)$ denotes the $n \times n$ diagonal matrix with eigenvalues 
$ \lambda_1 (V), \dots ,\lambda_n (V)$ and $I$ stands for the $n \times n$ identity matrix. The function $F ( V, \wcs, \xi, \e) \wcs$ takes values in the space of the $n \times n$ matrices and comes from the quadratic terms in the second equation of~\eqref{S3: autonomous ODE}.
Hence, by choosing a sufficienly small constant $\delta$ in~\eqref{e:bds}, we have the bound
\be
\label{e:F}
     | F ( V, \wcs, \xi, \e)   | \leq  K |\wcs|
\ee
for some suitable constant $K>0$. By rewriting~\eqref{e:der2}, we infer 
\be
\label{e:n}
           \Big(  C + \pcs \Big) \wcs' = \Big( \mathrm{Diag} \big( \lambda_1 , \dots ,\lambda_n \big)- \xi I_n \Big) \pcs \wcs +
           F\wcs - B    \Psi   \wcs -  \e ( \partial_{\xi}  \pcs  ) \wcs
\ee
and by relying on~\eqref{e:columns} and~\eqref{e:B}, we deduce that 
$$
     \Big(  C + \pcs \Big)   ( \bar U,  \vec 0, 0 ,0) = \Big( R_1(\bar U) \vert \dots \vert R_{k} (\bar U) \Big), 
$$
hence all the columns of this matrix are linearly independent at $ ( \bar U,  \vec 0, 0 ,0)$. By continuity, the columns of the matrix $  C + \pcs $ computed at $( V,  \wcs, \xi , \e)$ are linearly independent, provided that $( V,  \wcs, \xi , \e)$ satisfies~\eqref{e:bds} for a sufficiently small constant $\delta$. Hence, there exists a function $L ( V,  \wcs, \xi , \e)$ taking values in the space of the $k \times n$ matrices and satisfying 
$$
      L ( V,  \wcs, \xi , \e)  \Big(  C + \pcs \Big) ( V,  \wcs, \xi , \e) \equiv  I_{k} \;.
$$ 
By multiplying~\eqref{e:n} on the left by $L$, we arrive at 
$$
    \wcs' = \Big( \Lambda - \xi \Upsilon \Big)  \wcs , 
$$ 
where 
$$
   \Lambda  ( V,  \wcs, \xi , \e) = L  \mathrm{Diag} \big( \lambda_1 , \dots \lambda_n \big) \pcs +  L F - L B \Psi  - \e L ( \partial_{\xi}  \pcs  ) 
$$
 is a function taking values in the space of $k \times k$ matrices and satisfying 
$$
     \Lambda  ( \bar U,  \vec 0, 0 , 0) = \mathrm{Diag} \big( \lambda_1 (\bar U) , \dots ,\lambda_{k} (\bar U) \big). 
$$ 
To get the above equality, we use properties~\eqref{e:columns},~\eqref{e:B} and~\eqref{e:F}. Finally,
$
    \Upsilon ( V,  \wcs, \xi , \e) = L \pcs 
$
takes values in the space of $k \times k$ matrices and by relying on~\eqref{e:columns} and~\eqref{e:B}, we deduce that $\Upsilon( \bar U,  \vec 0, 0 , 0) $ is the $k\times k$ identity matrix. 
This completes the proof of Lemma~\ref{l:redu}. 
\end{proof}

\subsection{Construction of the approximating sequence} 
\label{sss:app}
Here, we construct a solution to the ODE~\eqref{e:redu} derived in the previous subsection. Let  $\{ \bar U^\e \} \subseteq \R^n$ be a family of data 
as in the statement of Theorem~\ref{t:main}, namely $\bar U^\e \to \bar U$ as $\e \to 0^+$.

To begin with, we fix a positive constant $ \bar \e >0$ and a vector $\bwcs \in \R^{k}$. Roughly speaking, the goal is to construct a solution to~\eqref{e:redu}, hence lying on $\mathcal{M}^{cs}$, such that $\wcs (0) = \bwcs$, 
$\e(\zeta) = \bar \e$,  $\lim_{\zeta \to + \infty} V(\zeta) = \bar U^\e$ and $\xi(0)=0$. However, in this problem a technical difficulty arises since by imposing $\xi(0) =0$, system~\eqref{e:redu} yields   $\xi (\zeta) = \bar \e \zeta$, whereas the manifold $\mathcal M^{cs}$ is defined for $|\xi| \leq \delta$ (see formula~\eqref{e:bds}). To overcome this difficulty, we preserve this bound $|\xi| \leq \delta$ by restricting to the interval 
$\zeta \in [0,  \delta / \bar \e \, ] $ and replacing the condition $\lim_{\zeta \to + \infty} V(\zeta) =  \bar U^\e$ with $V (\delta / \bar \e) =  \bar U^\e$. Also, to simplify the notation, we use $\e$ instead of $\bar \e$ and we write $\Ve$ and $\wcse$ to emphasize the dependence of the first two components of the solution to~\eqref{e:redu} on the parameter $\e$. After applying the variation of constants formula to~\eqref{e:redu}, we obtain the fixed point problem
\begin{equation}
\label{e:app:sy}
\left\{
\begin{array}{ll}
          \Ve(\zz) = \displaystyle \bar U^\e + \int_{\delta / \e}^{\zz} \Psi \Big( \Ve(y), \wcse(y), \e y, \e \Big) \wcse (y) \, dy\;,  \\ \\
          \wcse (\zz) = \displaystyle \exp{\Big( \bar \Lambda \zz - \e \zz^2 / 2 \,I_k \Big) } \bwcs + \\
          \qquad \qquad +  \displaystyle \int_0^{\zz} 
           \exp{\Big( \bar \Lambda (  \zz - y)  - I_k \e \zz^2 / 2  +  I_k \e y^2 / 2   \Big) }  
           \Bigg\{ \Bigg[ \Lambda \Big( \Ve(y), \wcse(y), \e y, \e \Big) - 
            \bar \Lambda \Bigg] + \\ \qquad \qquad -
            \e y \Bigg[ \Upsilon \Big( \Ve(y), \wcse(y), \e y, \e \Big) - I_k 
           \Bigg] \Bigg\}  \wcse(y) \, dy,   \\
\end{array}
\right.
\end{equation}
where $\bar \Lambda=\Lambda(\bar U, \vec 0, 0, 0)$. We recall that $\bar \Lambda$ is a diagonal matrix having only negative eigenvalues and $\Upsilon (\bar U, \vec 0, 0, 0) = I_k$. The following lemma provides  existence and uniqueness results for the solution of~\eqref{e:app:sy}. 

\begin{lemma}
\label{l:contra}
            Let $\Ve$ and $\wcs^\e$ be as in the statement of Lemma~\ref{l:redu}.
            If the constant $\delta$ in the statement of Lemma~\ref{l:redu} is sufficiently small, 
then we can find constants $c_V$, $\theta$ and $c_W$ which do not depend on $\e$ and satisfy the following properties: 
${0< c_V, \theta, c_W< 1}$ and for every sufficiently small $\e >0$ and for $\bwcs \in \R^{k}$ with 
            \be
            \label{e:theta}
            |\bwcs|\leq \theta \delta, 
            \ee
            we have that system~\eqref{e:app:sy} admits a unique solution satisfying 
            \begin{equation}
            \label{e:bd}
                |\Ve (\zz) -  \bar U| \leq c_V  
                 \delta  \quad \textrm{and} \quad  |\wcs^\e (\zz) | \leq c_W \exp \big(- c \,  \zz / 2 \big) \delta \;, 
            \ee
for every $\zz \in [0,  \delta/ \e ]$. 
\end{lemma}
\begin{proof}
The proof relies on a standard fixed point argument based on the Contraction Map Theorem. So here, we provide only a sketch of the proof. 

Fix $\e$ and $\bwcs$ as above and define the spaces
$$
    X^{\e}_V : = \left\{ V^\e \in \mathcal C^0 \big( [0,  \delta/ \e ] \, ;  \R^n \big)\;  \textrm{such that $|V ^\e(\zz) - \bar U| \leq c_V \delta $ for every $\zz\in [0,  \delta/ \e ]$}  \right\}
$$
and 
$$
      X^{\e}_W : = \left\{ \wcs^\e  \in \mathcal C^0 \big( [0,  \delta / \e]  \, ;  \R^{k}\big):
       \, |\wcs^\e (\zz) | \leq c_W \delta \exp \big(-c \, \zz / 2 \big) 
       \textrm{for every $\zz\in [0,  \delta/ \e ]$}  \right\}.
$$
The constants $c_V$ and $c_W$ are to be determined. By equipping $X_V$ and $X_W$ with the norms
$$
    \| V^\e \|_{V} : = \eta \| V^\e \|_{ \mathcal C^0 ( [0,  \delta / \e] )} \quad \textrm{and} \quad \| \wcs^\e \|_{W} : = \sup_{ \zz \in  [0,  \delta / \e] }  \exp \big( c \,  \zz / 2 \big) |\wcs^\e(\zz) |,
$$
we obtain two closed metric spaces. In the previous expression, $\eta$ denotes a real constant to be determined and $c$ is given by~\eqref{e:i:sesp}.  

Now, we define the map 
$$
     T^{\e} : X^{\e}_V \times X^{\e}_W \to \mathcal C^0 \big( [0,  \delta/ \e ]  \, ;  \R^n \big)  \times \mathcal C^0 \big( [0,  \delta/ \e ]  \, ;  \R^{k} \big) 
$$
by setting its components $T^{\e}_1(\Ve, \wcse)$ and $T^{\e}_2 (\Ve, \wcse)$ equal to the first and the second component of the right hand side of~\eqref{e:app:sy}, respectively.

By direct computations, one can show that it is possible to choose the constants $c_W$, $c_V$, $\theta$ and $\eta$ in such a way that $T^{\e}$ takes values in the set $X^{\e}_V \times X^{\e}_W $ and is, moreover, a strict contraction with respect to both the variables $\Ve$ and $\wcse$. We recall that $\theta$ is the same constant as in~\eqref{e:theta}. The values of 
$c_W$, $c_V$, $\theta$ and $\eta$ depend on $\delta$, but are independent of $\e$. By applying the Contraction Map Theorem, we conclude that~\eqref{e:app:sy} admits a unique solution belonging to $X^{\e}_V \times X^{\e}_W$.

To give a flavour of the construction, we choose to show that 
\be 
\label{e:t1}
     \| T^{\e}_1 (\Ve_1 , \wcsu ) - T^{\e}_1  (\Ve_2 , \wcsd )  \|_V \leq \frac{1}{2}  \big( \|\Ve_1 - \Ve_2  \|_V + \| \wcsu - \wcsd \|_W \big) 
\ee
for $(V_i^\e, W_{cs\,i}^e)\in X_V^\e\times X_W^\e$, $i=1,2$.
Indeed, we estimate 
\begin{equation*}
\begin{split}
                    | T^{\e}_1 (\Ve_1 ,& \wcsu ) -  T^{\e}_1  (\Ve_2 , \wcsd )  | (\zz) \\
             & \leq \int_{0}^{\delta / \e} \Big|  \Big[ \Psi \big( \Ve_1(y), \wcsu(y), \e y, \e \big) -   \Psi \big( \Ve_2 (y), \wcsd(y), \e y, \e \big)\Big] \wcsu(y)   \Big|  \, dy \, +  \\
             & \quad +\int_{0}^{\delta / \e}  \Big| \Psi \big( \Ve_2 (y), \wcsd(y), \e y, \e \big) \Big[ \wcsu(y) - \wcsd(y) \Big] \Big| dy    \leq   \\           
             & \leq \int_{0}^{\delta / \e}  \tilde{L} \Big( |\Ve_1 - \Ve_2 | (y) + |\wcsu - \wcsd | (y) \Big) c_W \delta  \exp \big( - c \,  y / 2 \big)  \, dy + \\
             & \quad + \int_{0}^{\delta / \e}   M |\wcsu - \wcsd | (y) \, dy\;, \\
\end{split}
\end{equation*}
where $\tilde{L}$ denotes a Lipschitz constant of the function $\Psi$ with respect to both the variables $V^\e$ and $\wcs^\e$, while $M$ is a uniform bound for $\Psi$ in the neighborhood of $(\bar U, \vec 0, 0, 0)$ given by~\eqref{e:bds}, namely  
\begin{equation}
\label{e:M}
     |   \Psi \big( V^\e , \wcs^\e , \xi, \e \big)  | \leq M\;, 
\end{equation} 
for every $(V^\e, \wcs^\e, \xi, \e)$ such that $|V^\e - \bar U| \leq \delta$, $|\wcs^\e| \leq \delta$, $|\xi| \leq \delta$ and $|\e| \leq \delta$.
By using the estimate
$$
     |\wcsu - \wcsd | (y) \leq   \exp \big( - c \,  y / 2 \big) \| \wcsu - \wcsd \|_W  \quad \textrm{for  $0 \leq y \leq \delta / \e$}, 
$$
we arrive at 
\begin{equation*}
\begin{split}
    \|  T^{\e}_1 (\Ve_1 , \wcsu ) -   T^{\e}_1  (\Ve_2 , \wcsd )  \|_{V} \leq \frac{2}{c} \Big[  &\tilde{L} c_W \delta \| \Ve_1 - \Ve_2 \|_V \\
    &+ ( \tilde{L} c_W \delta \eta + M \eta)  \| \wcsu - \wcsd \|_W \Big], 
\end{split}
\end{equation*}
which implies~\eqref{e:t1} provided that $\delta$ and $\eta$ are sufficiently small. 
\end{proof}

\subsection{The structure of the limit} 
\label{sus:limit}
We now establish the convergence result and analyze the limit.
\begin{lemma}
\label{l:asar}
         Let the hypotheses of Lemma~\ref{l:contra} be satisfied, let
         $\delta$, $c_V$ and $c_W$ be the same constants as in the statement of 
         Lemma~\ref{l:contra} and denote by $(\Ve, \wcse)$ the solution of 
            system~\eqref{e:app:sy} satisfying  bounds~\eqref{e:bd}. 
            Then, as $\e \to 0^+$, the family $\{(\Ve, \wcse)\}$ converges uniformly on compact subsets of $[0, + \infty[$. The limit $V^0$ is the unique solution of the ordinary differential equation
            \be
            \label{e:lim:sy}
                     (V^0)'' = A({V^0}) (V^0)'
            \ee 
            satisfying the conditions 
             \be 
            \label{e:lim:cd}
                       (V^0)' (0) = \bwcs \quad \textrm{and} \quad \lim_{\zz \to + \infty} V^0 (\zz) = \bar U 
            \ee
           and the bounds 
            \be 
            \label{e:lim:bd1}
                         |(V^0)' (\zz) | \leq c_W \exp \big( - c \,  \zz / 2 \big) \delta 
                         \qquad 
                           |V^0 (\zz) - \bar U| \leq  c_V \delta 
                         \quad \textrm{for every $\zz \in [0,  + \infty[$}.
            \ee
            Moreover, $V^0$ satisfies 
            \be
            \label{e:lim:bd2}
                         | V^0 (\zz) - \bar U| \leq \tilde c_V \delta \exp \big( - c \,  \zz / 2 \big) \quad \textrm{for every $\zz \in [0,  + \infty[$}, 
            \ee
             for a suitable constant $\tilde c_V >0$ depending only on $\bar U$, $c_W$, $c$ and on the upper bound $M$ given by~\eqref{e:M}.
\end{lemma}
\begin{proof}
We proceed in three steps.

\noindent
{\sc Step 1:} First, we establish compactness of the sequence $\{(\Ve,\wcse)\}$. 
Let $M$ denote a constant such that
\be
\label{e:M2}
        |   \Psi \big( V^\e , \wcs^\e , \xi, \e \big)  |, \;  
        |\Lambda \big( V^\e , \wcs^\e , \xi, \e \big)   |, \; 
         | \Upsilon \big( V^\e , \wcs^\e , \xi, \e \big)   | 
         \leq  M \quad  
\ee
for every $|V^\e - \bar U| \leq \delta$, $|\wcs^\e| \leq \delta$, $|\xi| \leq \delta$ and $|\e| \leq \delta$ and fix a compact subset $[0, b]$ of $[0,+\infty[$. 
By relying on~\eqref{e:bd}, we get that $\Ve$ satisfies the bounds 
\be
\label{e:b:V}
      \| \Ve - \bar U\|_{\mathcal C^0 ([0, b])}  \leq c_V  \delta \quad \textrm{and} \quad 
      \|  (V^{\e})' \|_{\mathcal C^0 ([0, b])}  =  \|  \Psi \wcse \|_{\mathcal C^0 ([0, b])} \leq M c_W \delta 
\ee
for every $0<\e \leq \delta / b$. Also, $\wcse$ satisfies 
\be
\label{e:b1:w}
    \| \wcse \| _{\mathcal C^0 ([0, b])}  \leq c_W  \delta 
\ee
and 
\be
\label{e:b2:w}
      \big\|  (\wcse)'  \big\|_{\mathcal C^0 ([0, b])}  =  
      \big\| \big( \Lambda - \e \zz \Upsilon \big) \wcse  \big\|_{\mathcal C^0 ([0, b])} \leq 
     M ( 1+\delta)  c_W \delta\;.  
\ee
From~\eqref{e:b:V}--\eqref{e:b2:w}, we get that the family $\{(\Ve, \wcse)\}$ satisfies all the hypotheses of Ascoli-Arzel\`{a}'s Theorem. Hence, for any given sequence $\e_n \to 0^+$, there exists a subsequence $\{\e_{n_k}\}$ and a pair $(V^0, \wcsz)$ such that 
$$
    V^{\e_{n_k}} \to V^0 \quad \textrm{and} \quad \wcs^{\e_{n_k}} \to \wcsz,
$$  
uniformly on $[0, b]$ as $\e_{n_k} \to 0^+$. We then consider a sequence $\{b_j\}$ of points such that $b_j \to + \infty$ and apply the previous argument on any interval $[0, b_j]$. By relying on a standard diagonalization procedure, we conclude that for any sequence $\e_n \to 0^+$ there exists a subsequence $\e_{n_k}$ and a pair $(V^0, \wcsz)$ such that 
$$
    V^{\e_{n_k}} \to V^0 \quad \textrm{and} \quad \wcs^{\e_{n_k}} \to \wcsz 
$$  
uniformly on the compact subsets of $[0, +\infty[$ as $\e_{n_k} \to 0^+$. In particular, $\{( V^{\e_{n_k}},  \wcs^{\e_{n_k}} )\}$ converges pointwise on $[0, + \infty[$ and hence 
from bounds~\eqref{e:bd} we deduce
\be
\label{e:bd0}
              | V^0 (\zz) -\bar U |\leq c_V  
                 \delta  \quad \textrm{and} \quad  |\wcs^0 (\zz) | \leq c_W \exp \big( - c \,  \zz / 2 \big) \delta \quad \textrm{for every $\zz \in [0,  + \infty[ $.}
\ee
{\sc Step 2:} Next, we prove that the limit $(V^0, \wcsz)$ is independent of the subsequence $\{\e_{n_k}\}$. By applying Lebesgue's Dominated Convergence Theorem to~\eqref{e:app:sy}, we get the identities
\begin{equation}
\label{e:app:sy0}
\left\{
\begin{array}{ll}
          V^0(\zz) = \displaystyle \bar U + \int_{+ \infty}^{\zz} \Psi \Big( V^0(y), \wcsz(y), 0, 0 \Big) \wcsz (y) \, dy  \\ \\
          \wcsz (\zz) = \displaystyle \exp{ \big( \bar \Lambda \zz \big) } \bwcs + 
          \displaystyle \int_0^{\zz} 
           \exp{ \big( \bar \Lambda (  \zz - y)  \big)}  
          \Big[ \Lambda \Big( V^0(y), \wcsz(y), 0, 0 \Big) - \bar \Lambda \Big] \wcsz(y) dy\;.  
\end{array}
\right.
\end{equation}
By relying on a  fixed point argument similar to the one in the proof of Lemma~\ref{l:contra}, we get that the solution $(V^0, \wcsz)$ of~\eqref{e:app:sy0} satisfying bounds~\eqref{e:bd0} is unique.  

\noindent
{\sc Step 3:} Last, we establish bound~\eqref{e:lim:bd2}. By relying on the representation formula~\eqref{e:app:sy0} and on estimates~\eqref{e:M2} and~\eqref{e:bd0}, we deduce that, for every $\zz \in [0, + \infty[$,
$$
     |V^0 (\zz) - \bar U| \leq \frac{2M c_W}{c}  \exp \big( - c \,  \zz / 2 \big)  \delta := \tilde c_V  \exp \big( - c \,  \zz / 2 \big)  \delta. 
$$
Finally, by comparing~\eqref{e:app:sy0} with~\eqref{S3: autonomous ODE} and~\eqref{e:redu}, we obtain that $V^0$ is a solution of the ODE~\eqref{e:lim:sy} satisfying conditions~\eqref{e:lim:cd}.
\end{proof}

\subsection{Conclusion of the proof of Theorem~\ref{t:main}}
\label{sus:conclu} Here, we complete the proof of Theorem~\ref{t:main}. 

We first denote by $V^s (\bar U)$ and $\mathcal{M}^s (\bar U)$ the stable space and stable manifold, respectively, for system~\eqref{e:i:1o} corresponding to the equilibrium point $(\bar U, \vec 0) $. Then, we change coordinates using matrix~\eqref{e:columns}. More precisely, every point $(V,W)\in V^s (\bar U)$ can be written in the form $W = \Phi_{cs} (\bar U, 0, 0, 0) \bwcs$ for some $\bwcs \in \R^k$. Next, we let $\theta$ and $\delta$ be as in Lemma~\ref{sss:app} and for given $\bwcs \in \R^k$ satisfying $|\bwcs| \leq \theta \delta$, we consider the unique solution $(V^0,\wcs^0)$ of~\eqref{e:app:sy0} satisfying bounds~\eqref{e:bd0}. Moreover, by varying $\bwcs$, we define the set 
$$
\mathcal M \doteq \Big\{ \big(V^0(0), \Phi_{cs} (\bar U, 0, 0, 0) \bwcs \big): \,\bwcs\in\R^k,\quad |\bwcs|\le \theta\delta \Big\} \subseteq \R^n \times \R^n.
$$ 

Now, we claim that the set $\mathcal M$ coincides with the stable manifold $\mathcal M^s (\bar U)$. To prove this claim, we first observe that system~\eqref{e:i:1o} is equivalent to system~\eqref{S3: autonomous ODE} provided that~\eqref{S3: autonomous ODE} is restricted to the invariant space $\{(V, W, \xi, \e): \; \xi = \e =0 \}\subseteq\R^n\times\R^n\times\R\times\R$. Hence, $\{ (V, W, \xi, \e): \; (V, W) \in \mathcal M^s (\bar U), \;   \xi = \e =0  \}\subseteq\mathcal M^{cs}$ for any given center-stable manifold $\mathcal M^{cs}$ of~\eqref{S3: autonomous ODE} corresponding to the equilibrium point $(\bar U, 0, 0, 0) $ and we recall that, by Lemma~\ref{l:redu}, solutions lying on $\mathcal M^{cs}$ satisfy~\eqref{e:redu}. Thus, to study the stable manifold $\mathcal M^s (\bar U)$ of~\eqref{e:i:1o}, it suffices to study the stable manifold to system~\eqref{e:redu}.
By recalling the proof of the Stable Manifold Theorem (see for example Perko~\cite[pp. 104-108]{Perko}) and the fact that $(V^0, \wcsz)$ satisfies~\eqref{e:bd0} and~\eqref{e:app:sy0} we finally establish that $\mathcal M=\mathcal M^s (\bar U)$. By taking $U_b=V^0(0)$ and by relying on Lemma~\ref{l:asar}, the proof of Part ($2$) of Theorem~\ref{t:main} is complete. Part ($1$) follows easily from Part ($2$) by defining $Q^\e(\xi)=\Ve( \xi/\e)$ for $\xi\in[0,\delta]$.

\section*{Acknowledgements}       
The authors would like to thank Professor Constantine Dafermos for proposing this project and Professors Stefano Bianchini and Athanasios Tzavaras for valuable discussions. Part of this work was done when Spinolo was supported by the Centro De Giorgi of Scuola Normale Superiore, Pisa, Italy, through a research contract. Christoforou was partially supported by the Start-Up Fund 2011-13 from University of Cyprus.  
\bibliography{biblio_t}

\def\cprime{$'$}
\begin{thebibliography}{10}

\bibitem{AnBia}
F.~Ancona and S.~Bianchini.
\newblock Vanishing viscosity solutions of hyperbolic systems of conservation
  laws with boundary.
\newblock In {\em ``{WASCOM} 2005''---13th {C}onference on {W}aves and
  {S}tability in {C}ontinuous {M}edia}, pages 13--21. World Sci. Publ.,
  Hackensack, NJ, 2006.

\bibitem{Andreianov}
B.~P. Andreianov.
\newblock The {R}iemann problem for {$p$}-systems with continuous flux
  function.
\newblock {\em Ann. Fac. Sci. Toulouse Math. (6)}, 8(3):353--367, 1999.

\bibitem{Bia:riemann}
S.~Bianchini.
\newblock On the {R}iemann problem for non-conservative hyperbolic systems.
\newblock {\em Arch. Ration. Mech. Anal.}, 166(1):1--26, 2003.

\bibitem{BiaBrevv}
S.~Bianchini and A.~Bressan.
\newblock Vanishing viscosity solutions of nonlinear hyperbolic systems.
\newblock {\em Ann. of Math. (2)}, 161(1):223--342, 2005.

\bibitem{BiaSpi}
S.~Bianchini and L.V. Spinolo.
\newblock The boundary {R}iemann solver coming from the real vanishing
  viscosity approximation.
\newblock {\em Arch. Ration. Mech. Anal.}, 191(1):1--96, 2009.

\bibitem{Bre:book}
A.~Bressan.
\newblock {\em {Hyperbolic systems of conservation laws. The one-dimensional
  Cauchy problem}}, volume~20 of {\em Oxford Lecture Series in Mathematics and
  its Applications}.
\newblock Oxford University Press, Oxford, 2000.

\bibitem{Bre:notes}
A.~Bressan, D.~Serre, M.~Williams, and K.~Zumbrun.
\newblock {\em Hyperbolic systems of balance laws}, volume 1911 of {\em Lecture
  Notes in Mathematics}.
\newblock Springer, Berlin, 2007.
\newblock Lectures given at the C.I.M.E. Summer School held in Cetraro, July
  14--21, 2003. Edited and with a preface by Pierangelo Marcati.

\bibitem{ChristoforouSpinolo}
C.~Christoforou and {L.V.} Spinolo.
\newblock A uniqueness criterion for viscous limits of boundary {Riemann}
  problems.
\newblock {\em J. Hyperbolic Differ. Equ}.
\newblock To appear. Also arXiv:1007.3931v1.

\bibitem{D1}
C.~M. Dafermos.
\newblock Solution of the {R}iemann problem for a class of hyperbolic systems
  of conservation laws by the viscosity method.
\newblock {\em Arch. Rational Mech. Anal.}, 52:1--9, 1973.

\bibitem{D}
C.~M. Dafermos.
\newblock {\em Hyperbolic conservation laws in continuum physics}, volume 325
  of {\em Grundlehren der Mathematischen Wissenschaften [Fundamental Principles
  of Mathematical Sciences]}.
\newblock Springer-Verlag, Berlin, third edition, 2010.

\bibitem{DalMasoLeFlochMurat}
G.~Dal~Maso, P.~G. LeFloch, and F.~Murat.
\newblock Definition and weak stability of nonconservative products.
\newblock {\em J. Math. Pures Appl. (9)}, 74(6):483--548, 1995.

\bibitem{Gis}
M.~Gisclon.
\newblock \'{E}tude des conditions aux limites pour un syst\`eme strictement
  hyperbolique, via l'approximation parabolique.
\newblock {\em J. Math. Pures Appl. (9)}, 75(5):485--508, 1996.

\bibitem{GisSerre}
M.~Gisclon and D.~Serre.
\newblock \'{E}tude des conditions aux limites pour un syst\`eme strictement
  hyperbolique via l'approximation parabolique.
\newblock {\em C. R. Acad. Sci. Paris S\'er. I Math.}, 319(4):377--382, 1994.

\bibitem{Glimm}
J.~Glimm.
\newblock Solutions in the large for nonlinear hyperbolic systems of equations.
\newblock {\em Comm. Pure Appl. Math.}, 18:697--715, 1965.

\bibitem{HoldenRisebro}
H.~Holden and N.~H. Risebro.
\newblock {\em Front tracking for hyperbolic conservation laws}, volume 152 of
  {\em Applied Mathematical Sciences}.
\newblock Springer-Verlag, New York, 2002.

\bibitem{JoLeF:sv}
K.~T. Joseph and P.~G. LeFloch.
\newblock Boundary layers in weak solutions of hyperbolic conservation laws.
  {II}. {S}elf-similar vanishing diffusion limits.
\newblock {\em Commun. Pure Appl. Anal.}, 1(1):51--76, 2002.

\bibitem{JoLeF:nc}
K.~T. Joseph and P.~G. LeFloch.
\newblock {Singular limits for the Riemann problem: general diffusion,
  relaxation, and boundary conditions}.
\newblock In {\em New analytical approach to multidimensional balance laws}. O.
  Rozanova ed., Nova Press, 2007.

\bibitem{kal}
A.~S. Kala{\v{s}}nikov.
\newblock Construction of generalized solutions of quasi-linear equations of
  first order without convexity conditions as limits of solutions of parabolic
  equations with a small parameter.
\newblock {\em Dokl. Akad. Nauk SSSR}, 127:27--30, 1959.

\bibitem{KHass}
A.~Katok and B.~Hasselblatt.
\newblock {\em Introduction to the modern theory of dynamical systems},
  volume~54 of {\em Encyclopedia of Mathematics and its Applications}.
\newblock Cambridge University Press, Cambridge, 1995.
\newblock With a supplementary chapter by Katok and Leonardo Mendoza.

\bibitem{lax}
P.~D. Lax.
\newblock Hyperbolic systems of conservation laws. {II}.
\newblock {\em Comm. Pure Appl. Math.}, 10:537--566, 1957.

\bibitem{LeflochTzavaras}
P.~G. LeFloch and A.~E. Tzavaras.
\newblock Representation of weak limits and definition of nonconservative
  products.
\newblock {\em SIAM J. Math. Anal.}, 30(6):1309--1342 (electronic), 1999.

\bibitem{Liu:rie}
T.~P. Liu.
\newblock The {R}iemann problem for general systems of conservation laws.
\newblock {\em J. Differential Equations}, 18:218--234, 1975.

\bibitem{Liu:adm}
T.~P. Liu.
\newblock The entropy condition and the admissibility of shocks.
\newblock {\em J. Math. Anal. Appl.}, 53(1):78--88, 1976.

\bibitem{Perko}
L.~Perko.
\newblock {\em Differential equations and dynamical systems}, volume~7 of {\em
  Texts in Applied Mathematics}.
\newblock Springer-Verlag, New York, third edition, 2001.

\bibitem{Serre:book}
D.~Serre.
\newblock {\em Systems of conservation laws. 1 and 2}.
\newblock Cambridge University Press, Cambridge, 1999.
\newblock Translated from the 1996 French original by I. N. Sneddon.

\bibitem{Tu1966}
V.~A. Tup{\v{c}}iev.
\newblock The problem of decomposition of an arbitrary discontinuity for a
  system of quasi-linear equations without the convexity condition.
\newblock {\em \u Z. Vy\v cisl. Mat. i Mat. Fiz.}, 6:527--547, 1966.

\bibitem{Tzavaras:JDE}
A.~E. Tzavaras.
\newblock Elastic as limit of viscoelastic response, in a context of
  self-similar viscous limits.
\newblock {\em J. Differential Equations}, 123(1):305--341, 1995.

\bibitem{Tz}
A.~E. Tzavaras.
\newblock Wave interactions and variation estimates for self-similar
  zero-viscosity limits in systems of conservation laws.
\newblock {\em Arch. Rational Mech. Anal.}, 135(1):1--60, 1996.

\end{thebibliography}
\end{document}